\documentclass[11pt, twoside]{article}
\usepackage{amssymb, amsmath, amsthm}
\usepackage[left=27mm, right=27mm, bottom=30mm]{geometry}
\usepackage{inputenc}
\usepackage{fancyhdr}
\usepackage{tikz}
\usetikzlibrary{positioning}
\usepackage{titling}
\usepackage{hyperref}
\predate{}
\postdate{}
\usepackage{lipsum}
\newtheorem{theorem}{Theorem}

\newtheorem{claim}{Claim}

\newtheorem{question}[theorem]{Question}
\newtheorem{conjecture}[theorem]{Conjecture}
\usepackage{titling}
\usepackage{xcolor}
\usepackage{placeins}
\usepackage{tikz}
\usepackage{verbatim}
\usepackage{titling}
\usepackage[T1]{fontenc}
\usepackage{currvita}
\usepackage{bbm}
\usepackage{enumitem}
\newtheorem{corollary}[theorem]{Corollary}
\newtheorem{claim2}{Claim}

\theoremstyle{definition}

\theoremstyle{remark}

\newcommand{\cP}{\mathcal{P}}

\newcommand{\cS}{\mathcal{S}}

\newcommand{\B}{\mathcal{B}}

\begin{document}
\newcommand{\Addresses}{
\bigskip
\footnotesize

\medskip

\noindent Tom\'a\v s~Fl\'idr, \textsc{Peterhouse, University of Cambridge, CB2 1RD, UK.}\par\noindent\nopagebreak\textit{Email address: }\texttt{tf388@cam.ac.uk}

\medskip

\noindent Maria-Romina~Ivan, \textsc{Department of Pure Mathematics and Mathematical Statistics, Centre for Mathematical Sciences, Wilberforce Road, Cambridge, CB3 0WB, UK.}\par\noindent\nopagebreak\textit{Email addresses: }\texttt{mri25@dpmms.cam.ac.uk}

\medskip

\noindent Sean~Jaffe, \textsc{Department of Pure Mathematics and Mathematical Statistics, Centre for Mathematical Sciences, Wilberforce Road, Cambridge, CB3 0WB, UK.}\par\noindent\nopagebreak\textit{Email address: }\texttt{scj47@cam.ac.uk}}

\pagestyle{fancy}
\fancyhf{}
\fancyhead [LE, RO] {\thepage}
\fancyhead [CE] {TOM\'A\v S FL\'IDR, MARIA-ROMINA IVAN AND SEAN JAFFE}
\fancyhead [CO] {OPTIMAL EMBEDDINGS OF POSETS IN HYPERCUBES}
\renewcommand{\headrulewidth}{0pt}
\renewcommand{\l}{\rule{6em}{1pt}\ }
\title{\Large{\textbf{OPTIMAL EMBEDDINGS OF POSETS IN HYPERCUBES}}}
\author{TOM\'A\v S FL\'IDR, MARIA-ROMINA IVAN AND SEAN JAFFE}
\date{ }
\maketitle
\begin{abstract}
Given a finite poset $\mathcal P$, the hypercube-height, denoted by $h^*(\mathcal P)$, is defined to be the minimum $h$ such that there exists a natural number $n$ for which the subsets of $[n]$ of size at most $h$ contain an induced copy of $\mathcal P$. The hypercube-width, denoted by $w^*(\mathcal P)$, is the smallest $w$ such that the subsets of $[w]$ of size at most $h^*(\mathcal P)$ contain an induced copy of $\mathcal P$. In other words, $h^*(\mathcal P)$ asks how `low' can a poset be embedded, and $w^*(\mathcal P)$ asks for the first hypercube in which such an `optimal' embedding occurs.

These notions were introduced by Bastide, Groenland, Ivan and Johnston in connection to upper bounds for the poset saturation numbers. While it is not hard to see that $h^*(\mathcal P)\leq |\mathcal P|-1$ (and this bound can be tight), the hypercube-width has proved to be much more elusive. It was shown by the authors mentioned above that $w^*(\mathcal P)\leq|\mathcal P|^2/4$, but they conjectured that in fact $w^*(\mathcal P)\leq |\mathcal P|$ for any finite poset $\mathcal P$.

In this paper we prove this conjecture. The proof uses Hall's theorem for bipartite graphs as a precision tool for modifying an existing copy of our poset.

\end{abstract}
\section{Introduction}
A \textit{poset} is short for a partially ordered set. In this paper all posets are finite. Although posets are generally abstract notions, their natural arena is in fact the power set equipped with the partial order given by set inclusion. Indeed, given a poset $\mathcal P$ with elements $\{p_1,p_2\dots,p_t\}$, and partial order $\preceq$, we can \textit{realise} $\mathcal P$ inside the hypercube $Q_t$ via the sets $A_i=\{j:p_j\preceq p_i\}$ for all $i\in[t]$. Of course, this embedding is not unique, not even inside $Q_t$. For example, if $\mathcal P$ is the antichain of size 6, then the above gives an embedding in $Q_6$ where all the elements of the poset are singletons, but we can also take all pairs of $\{1,2,3,4\}$ and obtain another copy of the antichain of size 6 in $Q_6$. One difference between these two embeddings is that the first sits `lower' in the hypercube than the other. So in general,  how can we 'optimally fit' a given poset $\mathcal P$ into a hypercube?

This question was investigated by Bastide, Groenland, Ivan and Johnston \cite{polynomial} in order to achieve a general upper bound for the induced saturation numbers for posets. As such, for a given poset $\mathcal P$, they have introduced the notions of \textit{hypercube-height} and \textit{hypercube-width}. To define them, we first need some standard notation. Given two integers $h\leq w$, we denote by $\binom{[w]}{\leq h}$ the induced subposet of the hypercube $Q_w$ consisting  of all the sets of size at most $h$, i.e the poset $Q_w$ restricted to the first $h + 1$ layers, $0,1,\dots,h$. 

For a poset $\mathcal P$, we define the \textit{hypercube-height} $h^*(\mathcal P)$ to be the minimum $h^* \in \mathbb{N}$ for which there exists $n \in \mathbb{N}$ such that $\binom{[n]}{\leq h^*}$ contains an induced copy of $\mathcal{P}$.

For a poset $\mathcal P$, we define the \textit{hypercube-width} $w^*(\mathcal P)$ to be the minimum $w^* \in \mathbb{N}$ such that there exists an induced copy of $\mathcal P$ in $\binom{[w^*]}{\leq h^*(\mathcal P)}$.

It is important to note that the two notions defined above are different from the usual height and width of $\mathcal P$, that is, from the size of the biggest chain and antichain, respectively. One example is the butterfly poset (two maximal elements both bigger than two minimal elements). The height of the butterfly is 2, but its hypercube-height is 3 since the first 3 layers of any hypercube are butterfly-free. Similarly, the width and the hypercube-width can be very different. For example, if $\mathcal P$ is a chain of size $k$, then its width is 1, but its hypercube-width is $k-1$.

But how big can the hypercube-height and hypercube-width be? The `canonical' embedding mentioned in the beginning immediately gives us that $h^*(\mathcal P)\leq |\mathcal P|$. In fact, one can easily modify this embedding, as shown in \cite{polynomial}, to obtain $h^*(\mathcal P)\leq|\mathcal P|-1$, which can even be tight, e.g., when $\mathcal P$ is a chain. What about the hypercube-width?

This notion is much harder to grasp as it is defined via the hypercube-height, which although we know how to upper bound well, we do not yet understand how to structurally tie it to the poset. Hypercube-width is not even a monotone property. For example, the antichain of size $\binom{k}{k/2}$ has hypercube-height 1 and hypercube-width $\binom{k}{k/2}$, but adding a chain of length $k/2$ which is less than all elements of the antichain gives a poset with hypercube-height $k/2$ and hypercube-width $k$. For any poset that one writes down, we always seem to have $w^*(\mathcal P)\leq |\mathcal P|$, which led to the following.

\begin{conjecture}[Conjecture 9 in \cite{polynomial}]
\label{mainconjecture}
For any finite poset $\mathcal P$ we have $w^*(\mathcal P)\leq |\mathcal P|$.
\end{conjecture}
It was shown in $\cite{polynomial}$ that $w^*(\mathcal P)\leq |\mathcal P|h^*(\mathcal P)$, and also that $w^*(\mathcal P)\leq |\mathcal P|^2/4$. Unfortunately, both lead, in general, to an upper bound that is quadratic in $|\mathcal P|$.

In this paper we prove Conjecture~\ref{mainconjecture}. We first focus on two-layered posets, and present in Section 2 a proof that is tailored for this type of posets -- this proof is guided by the natural intuition that some minimal elements of such an `optimal' embedding may be taken to be singletons. We emphasise that this proof only works for two-layered posets. 

In Section 3 we prove the conjecture in full generality. One may completely skip Section 2 as the proofs are independent -- we included it as it is a more `hands on' proof which can serve as a conceptual `warm-up' for the more general proof. Both results have at heart Hall's theorem for bipartite graphs, which helps us replace (in some sense) some sets of an already existing copy of our poset by singletons. For completeness, we end the introduction with Hall's theorem.
\begin{theorem}[\cite{Hall}]
Let $G$ be a finite bipartite graph with bipartite sets $X$ and $Y$ and edge set $E$. Two edges are disjoint if they do not share any vertex. An $X$-matching is a set of disjoint edges that covers every vertex in $X$. For a subset $W$ of $X$, we denote by $N(W)$ the neighbourhood of $W$. Then, there exists an $X$-matching if and only if for every subset $W$ of $X$ we have $|W|\leq |N(W)|$.
\end{theorem}
\section{Two-layered posets}
In this section we are looking at two-layered posets. We define a \textit{two-layered poset} to be a poset in which every element is either a maximal element, or a minimal element, and not both. We show that the hypercube-width of such a poset is always at most the size of the poset. Our proof uses Hall's theorem for bipartite graphs in order to modify a given copy of the poset in a hypercube, by replacing some minimal elements with singletons, and also modifying accordingly the maximal elements above the minimal elements we have modified. 

We mention that this approach only seems to work for two-layered posets and it is somewhat different from the general proof presented in Section 3. We include it here for a smooth and intuitive transition to the general proof, as well as for a broader understanding of `optimal' embeddings for two-layered posets.
\begin{theorem}
Let $\mathcal P$ be a two-layered poset. Then $w^*(\mathcal P)\leq |\mathcal P|$.
\end{theorem}
\begin{proof}
Suppose $\mathcal P$ is embedded in some $Q_n$ such that all sets of this embedding have size at most $h^*(\mathcal P)$. We will call this copy $\mathcal P$ for readability purposes. Let $\mathcal A=\{A_1,\dots,A_r\}$ be its maximal elements and $\mathcal B=\{{B}_1,\dots,B_k\}$ its minimal elements. If $k=1$, then the poset is $\mathcal V_r$, depicted below, for which trivially $h^*(\mathcal P)=1$, and $w^*(\mathcal P)=r=|\mathcal V_r|-1$. Therefore, we may assume that $k\geq 2$, which consequently implies that $\emptyset\notin\mathcal P$.
\begin{center}
\includegraphics[width=13em]{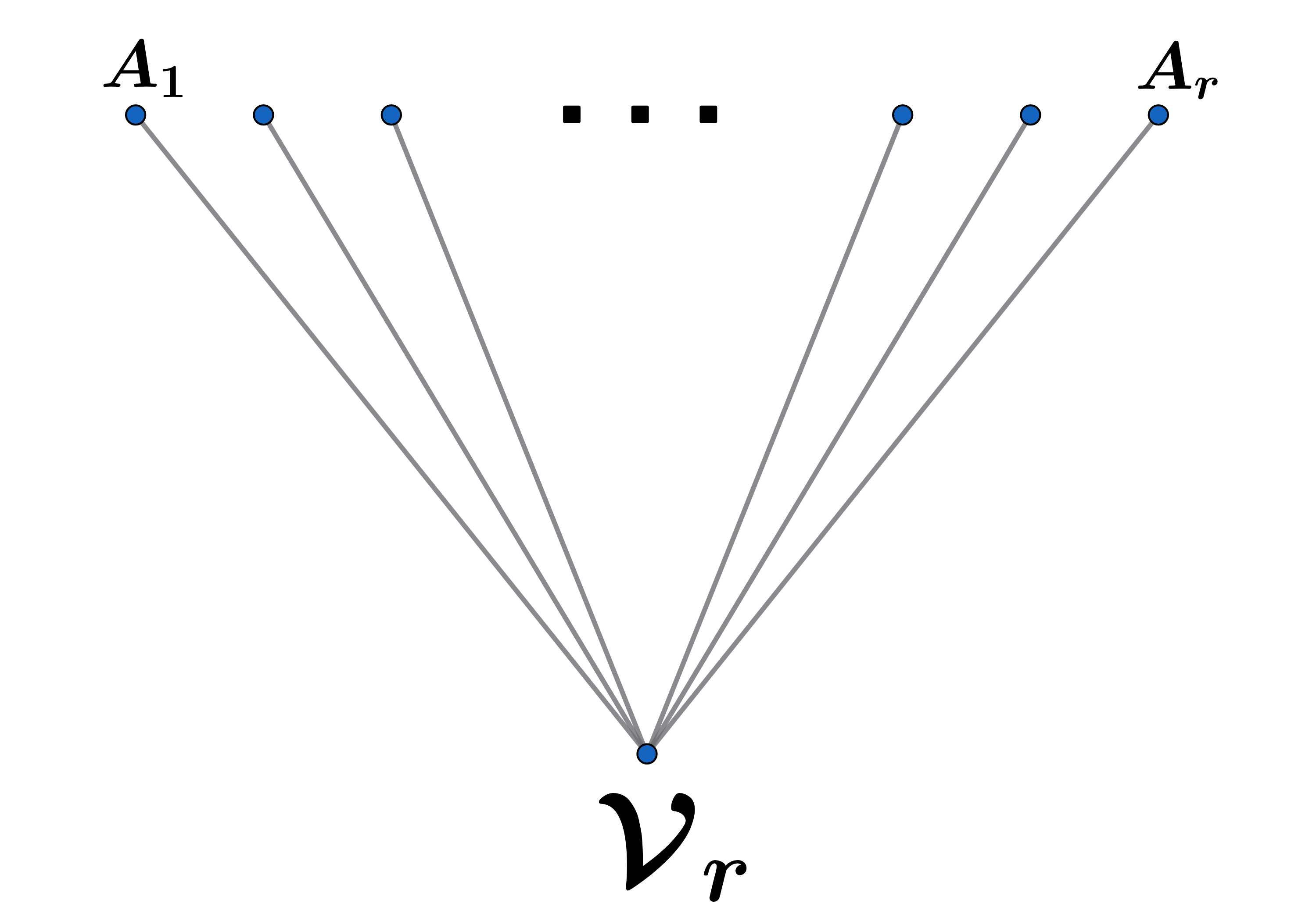}
\end{center}

Let $\mathcal G$ be all the elements of the ground set that appear in some $B_i$, i.e $\mathcal G = \bigcup_{i\in[k]}B_i$. Let $\mathcal X$ be a maximal size subset of minimal elements of $\mathcal P$ for which the size of their union is less than their number, i.e. $\mathcal X\subset\mathcal B$ such that $\left|\bigcup_{B\in\mathcal X} B\right| < |\mathcal X|$, and $|\mathcal X|$ is maximal. If no such set exists, we define $\mathcal X$ to be $\emptyset$. Let $\mathcal Y = \mathcal B\setminus \mathcal X$ and $\mathcal G' = \mathcal G \setminus \bigcup_{B\in\mathcal X} B$, the elements of the ground set that do not appear in any set of $\mathcal X$. 

If $\mathcal Y\neq\emptyset$, consider the bipartite graph with classes $\mathcal Y$ and $\mathcal G'$, and an edge between $B$ and $x$ if and only if $x\in\B$. Suppose that there is no matching from $\mathcal Y$ to $\mathcal G'$. Then, by Hall's theorem, there exists a non-empty subset $\mathcal S\subseteq\mathcal Y$ such that the number of neighbours of $\mathcal S$ is less than $|\mathcal S|$, or in other words $\left|\mathcal G' \cap \bigcup_{B\in\mathcal S} B\right|<|\mathcal S|$. By construction, $\mathcal S$ and $\mathcal X$ are disjoint, thus $|\mathcal S\cup\mathcal X| = |\mathcal S|+|\mathcal X|$. Moreover, by the definition of $\mathcal G'$, we have that $\left|\bigcup_{B\in\mathcal S\cup\mathcal X} B\right| = \left|\bigcup_{B\in\mathcal X} B\right| +\left|\mathcal G' \cap \bigcup_{B\in\mathcal S} B\right| < |\mathcal S|+|\mathcal X|$, which contradicts the maximality of $\mathcal X$, as $\mathcal S\neq\emptyset$.

Therefore $\mathcal B$ is the disjoint union of $\mathcal X$ and $\mathcal Y$, where  $\left|\bigcup_{B\in X} B\right| < |\mathcal X|$, and there exists a matching (injective function) $f:\mathcal Y\to \mathcal G'$. Note that this is vacously true in the case where $\mathcal Y=\emptyset$.

We will now modify the embedding by replacing the sets in $\mathcal Y$ with the singletons given by the matching, and the maximal elements by the union of the sets they must contain, plus some extra new singletons in order to ensure incomparability where necessary. In particular, we completely throw away all elements that are covered by the maximal sets, but not by any minimal set.

Let $a_1,\dots, a_r\in\mathbb N\setminus \mathcal G$ be $r$ distinct numbers. First, for all $i\in[r]$, we define $C_i' = \{f(B): B\in\mathcal Y, B\subset A_i\} \cup \bigcup_{B\in\mathcal X: B\subset A_i} B$. These are our candidates for the new maximal elements. However, they could be in $\mathcal X$, a singleton, or comparable. So, if $C_i'\in\mathcal X$, or $|C_i'|=1$ or if $C'_i\subseteq C_j'$, for some $i\neq j$, we define $C_i = C_i'\cup \{a_i\}$. Otherwise, we define $C_i=C_i'$.

\begin{claim} The family $\mathcal X \cup \{\{f(B)\}:B\in\mathcal Y\} \cup \{C_i: i\in[r]\}$ is an induced copy of $\mathcal P$, where the maximal elements are $C_i$ for $i\in[r]$, while the rest are the minimal elements.
\end{claim}
\begin{proof}
Consider the function $g:\mathcal P\rightarrow\mathcal X\cup\{\{f(B)\}:B\in\mathcal Y\}\cup\{C_i:i\in[r]\}$, given by $g(A_i)=C_i$, $g(B_i)=B_i$ if $B_i\in\mathcal X$ and $g(B_i)=\{f(B_i)\}$ if $B_i\in\mathcal Y$. We will show that $g$ is a poset isomorphism. In other words, we will show that $\mathcal X\cup\{\{f(B)\}:B\in\mathcal Y\}$ is an antichain of size $k$, $\{C_i:i\in[r]\}$ is an antichain of size $r$ and that they are distinct.  Moreover, $B_i\subset A_j$ if and only if $g(B_i)\subset g(A_j)$.

To begin with, since $\mathcal X\subseteq\mathcal B$, $\mathcal X$ is an antichain. By construction, the image of $f$ is a set of distinct singletons, hence also an antichain. Moreover, these singletons are in $\mathcal G'$, which is disjoint from any $B\in\mathcal X$. Therefore we have that $\mathcal X\cup\{\{f(B)\}:B\in\mathcal Y\}$ is an antichain of size $k$.

Next, suppose that $C_i\subseteq C_j$ for some $i\neq j$. Since by construction we have $C_j\subseteq \mathcal G \cup \{a_j\}$, we get that $a_i\notin C_j$, and consequently $a_i\notin C_i$, thus $C_i=C_i'$. By definition, this means that $C_i'\not\subseteq  C_j'$. This is a contradiction as $a_j\notin C_i$, and so $C_i'=C_i\subseteq C_j\setminus\{a_j\} = C_j'$. Therefore $\{C_i:i\in[r]\}$ is an antichain of size $r$.

Moreover, we cannot have $g(B_i)=C_j$ for any $i\in[k]$ and $j\in[r]$. This is because if $C'_j$ is in $\mathcal X$ or a singleton, we add a completely new singleton to $C'_j$ to make $C_j$. Therefore, the two antichains are disjoint.

Suppose now that $B_i\subset A_j$. Then by construction we have $g(B_i)\subseteq C_j'$ both in the case when $B\in\mathcal X$ or when $B\in\mathcal Y$. Therefore $g(B_i)\subset g(A_j)=C_j$.

Lastly, suppose that there exist $B_i$ and $A_j$ such that $B_i\not\subset A_j$ but $g(B_i)\subset g(A_j)=C_j$. If $B_i\in \mathcal X$, then $B_i\subset C_j$. Since $B_i$ does not contain $a_i$, we must have $B_i\subseteq C'_j=\{f(B): B\in\mathcal Y, B\subset A_j\} \cup \bigcup_{B\in\mathcal X: B\subset A_j}B$. Since $B_i$ is disjoint from $\{f(B): B\in\mathcal Y\}$, we get that $B_i\subseteq\bigcup_{B\in\mathcal X: B\subset A_j}B\subseteq A_j$, a contradiction. If $B_i\in\mathcal Y$, then $f(B_i)\in\mathcal G'$. However, since the only elements of $C_j$ that are in $\mathcal G'$ are $f(B)$ for $B\subset A_j$, and $f$ is an injection, we must have $B_i\subset A_j$, a contradiction. This finishes the claim.
\end{proof}
We now show that this new embedding is still `as low as possible'.
\begin{claim} We have that $\max\{|g(A)|:A\in \mathcal P\} = h^*(\mathcal P)$.
\end{claim}
\begin{proof}
It is enough to show that $|C_i|\leq |A_i|$ for all $i\in [r]$. Since $f(B)\in B$ for all $B\in\mathcal Y$, we clearly have $C_i'\subseteq A_i$. If $C_i' = C_i$, then we are done. Otherwise, $|C_i|=|C_i'|+1$ in which case either $C'_i\in\mathcal X$, $C'_i$ is a singleton or $C_i'\subseteq C_j'$ for some $j\neq i$. If $C'_i\in\mathcal X\subseteq\mathcal B$, then $C'_i\subset A_i$, and so $|C_i|\leq |A_i|$. If $C'_i$ is a singleton, since $A_i$ cannot be a singleton (as $\emptyset\notin\mathcal P$), we still have $|C_i|\leq|A_i|$. Finally, if $C'_i\subseteq C'_j\subseteq A_j$ for some $i\neq j$, then $C'_i\subset A_i$ because otherwise we would have $A_i\subseteq A_j$, a contradiction. Therefore $|C_i|\leq|A_i|$ in this case too, finishing the proof of the claim.
\end{proof}

Putting everything together, we get $$w^*(\mathcal P)\leq \left|\bigcup_{A\in\mathcal P} g(A)\right| \leq \left|\mathrm{Im}f \cup \bigcup_{B\in\mathcal X}B \cup \{a_i:i\in[r]\} \right|\leq|\mathcal Y| + |\mathcal X| + |\mathcal A| = |\mathcal P|.$$
\end{proof}

\section{The general case}
We now turn our attention to an arbitrary poset $\mathcal P$. The strategy is similar to the one in Section 2. Given a copy of $\mathcal P $ embedded in some $Q_n$, we use Hall's theorem to construct a matching from some sets of $\mathcal P$ to a carefully chosen subset of $[n]$. We then modify each set of the poset by keeping the elements that are not in that chosen subset of $[n]$, and adding the singletons, given by the matching, corresponding to the sets that are below it.
\begin{theorem}
\label{maintheorem}
Let $\mathcal P$ be a finite poset. Then $w^*(\mathcal P)\leq |\cP|$.
\end{theorem}
\begin{proof}
Suppose $\mathcal P$ is embedded in some $Q_n$ such that all sets of this embedding have size at most $h^*(\mathcal P)$. We will call this fixed copy $\mathcal P$ for readabilty purposes.

Let $\mathcal X$ be a subset of $\mathcal P$ of maximal size such that $|\bigcup_{A\in \mathcal X}A| < |\mathcal X|$. If no such subset exists, then we set $\mathcal X = \emptyset$. Let  also $\mathcal Y = \mathcal P\setminus\mathcal X$, $\mathcal G = \bigcup_{A\in \mathcal X}A$ and $\mathcal G' = (\bigcup_{A\in\mathcal P}A)\setminus \mathcal G$. 

We note that regardless of whether $\mathcal X=\emptyset$ or not, $|\mathcal G|\leq |\mathcal X|$.

Suppose first that $\mathcal Y\neq \emptyset$. In this case, we consider the bipartite graph with classes $\mathcal Y$ and $\mathcal G'$ and edges between an $A\in\mathcal Y$ and an $x\in\mathcal G'$ if and only if $x\in A$. Suppose that there is no matching from $\mathcal Y$ to $\mathcal G'$. Then, by Hall's theorem, there exists a non-empty subset $\mathcal S\subseteq \mathcal Y$ such that $\mathcal S$ has less than $|\mathcal S|$ neighbours. In other words, $|\mathcal G'\cap (\bigcup_{A\in \cS}A)| <|\cS|$. By construction, $\mathcal Y$ and $\mathcal X$ are disjoint, thus $\mathcal S$ and $\mathcal X$ are disjoint too. We therefore get that $|\mathcal S\cup\mathcal X| = |\mathcal S| + |\mathcal X|$. Moreover, by the definitions of $\mathcal G$ and $\mathcal G'$, $|\bigcup_{A\in \mathcal S\cup\mathcal X}A| = |\mathcal G| + |\mathcal G'\cap( \bigcup_{A\in \mathcal S}A)| < |\mathcal X| + |\mathcal S|$, contradicting the maximality of $\mathcal X$ since $\mathcal S\neq \emptyset$. Therefore there exists a matching from $\mathcal Y$ to $\mathcal G'$. In other words, there exists an injective function $f:\mathcal Y\to\mathcal G'$ such that $f(A)\in A$ for all $A\in\mathcal Y$. If $\mathcal Y=\emptyset$, then this is vacuously true.

We will now use this matching to modify the embedding of $\mathcal P$. Let $g$ be the function $g:\mathcal P\to Q_n$ such that $$g(A)=(A\cap \mathcal G) \cup \{f(B):B\in \mathcal Y, B\subseteq A\}.$$ Let $\mathcal P'$ be the image of $g$.

\begin{claim2} $\mathcal P'$ is an induced copy of $\mathcal P$.
\end{claim2}
\begin{proof}
It is trivial to see that if $A\subseteq A'$, then $g(A)\subseteq g(A')$. We will show that if $g(A)\subseteq g(A')$, then $A\subseteq A'$, which will show that $g$ is an injective order preserving map from $\mathcal P$ to $\mathcal P'$, finishing the proof of the claim.

Suppose that $g(A)\subseteq g(A')$ for some $A,A'\in\mathcal P$. This means that $(A\cap \mathcal G) \cup \{f(B):B\in \mathcal Y, B\subseteq A\}\subseteq (A'\cap \mathcal G) \cup \{f(B):B\in \mathcal Y, B\subseteq A'\}$. Since the image of $f$ is disjoint from $\mathcal G$ and $f$ is injective, we get that $A\cap \mathcal G \subseteq A'\cap \mathcal G$ and $\{B\in \mathcal Y:B\subseteq A\}\subseteq \{B\in \mathcal Y:B\subseteq A'\}$.

If $A\in \mathcal X$, then $A=A\cap \mathcal G \subseteq A'\cap \mathcal G \subseteq A'$. If $A\in\mathcal Y$, then $A\in \{B\in \mathcal Y:B\subseteq A\}\subseteq \{B\in\mathcal Y, B\subseteq A'\}$, which implies that $A\subseteq A'$, finishing the proof of the claim as explained above.
\end{proof}

By construction we have that $f(A)\in A$ for all $A\in \mathcal P$, and so $g(A)\subseteq A$ for all $A\in\mathcal P$. Consequently, we get $\max \{|A|:A\in\mathcal P'\} \leq \max \{|A|:A\in\mathcal P\} = h^*(\cP)$. Putting everything together, we get

$$
w^*(\cP) \leq\left|\bigcup_{A\in\mathcal P'}A\right| \leq |\mathcal G| + |\text{Im} f| \leq |\mathcal X| + |\mathcal Y| = |\mathcal P|.
$$
\end{proof}
\section{Final remarks}
We would like to mention that, since this question of how big the hypercube-width can be arose in \cite{polynomial} as a central tool in upper bounding the saturation number of an arbitrary poset. 

Given a finite poset $\mathcal P$, we say that a family of sets $\mathcal F
\subseteq\mathcal P([n])$ is \textit{$\mathcal P$-saturated} if it does not contain an induced copy of $\mathcal P$, but $\mathcal F\cup\{A\}$ contains such a copy for all $A\in \mathcal P([n])\setminus\mathcal F$. The size of the smallest such family is called the \textit{induced saturation number of $\mathcal P$}, denoted by $\text{sat}^*(n,\mathcal P)$. Theorem~\ref{maintheorem} gives the following result.
\begin{corollary}
\label{corollary}
Let $\mathcal P$ be a finite poset. Then $\text{sat}^*(n,\mathcal P)\leq 2n^{|\mathcal P|-1}$ for sufficiently large $n$.
\end{corollary}
Determining the saturation number, even for specific posets, is nowhere near a trivial task. In fact, these numbers have been determined for just a few posets -- see \cite{ferrara2017saturation, ivan2020saturationbutterflyposet, keszegh2021induced, bastide2024exact, diamondlinear, N}. Furthermore, they are all either bounded or linear in $n$, which is in fact the central conjecture of the field \cite{keszegh2021induced}. 

This result gets one step closer to it, lowering the previous upper bound of $2n^{|\mathcal P|^2/4-1}$ in \cite{polynomial}. However, since there exist posets such that $w^*(\mathcal P)=|\mathcal P|$, and the work in \cite{polynomial} is optimized to show that $\text{sat}^*(n,\mathcal P)\leq 2n^{w^*(\mathcal P)-1}$, Corollary~\ref{corollary} is the best upper bound that can be achieved with these techniques. Therefore, in order to show that the saturation number for any poset is at most linear, substantially new methods need to be developed.

Nevertheless, the question of how `low' a poset can be embedded in a hypercube is in itself an interesting one, regardless of its connection to poset saturation. There are two notions here that are closely linked, namely the hypercube-height and the hypercube-width. Whilst one might superficially think that the hypercube-height is understood, in the sense that $0\leq h^*(\mathcal P)\leq |\mathcal P|-1$ and both bounds can be achieved (by the single point poset and the chain, respectively), we do not know how to compute the hypercube-height of an arbitrary poset. In other words, what structural feature makes a poset have large hypercube-height? With this in mind, we ask the following natural question.
\begin{question}
For which finite posets $\mathcal P$ do we have that $h^*(\mathcal P)=|\mathcal P|-1$?
\end{question}
Moreover, the hypercube-height and the hypercube-width are very closely linked -- seemingly one cannot understand one without the other. This prompts us to asking the following question.
\begin{question}
Is there a tight inequality between $h^*(\mathcal P)$ and $w^*(\mathcal P)$ that holds for any finite poset $\mathcal P$, and only involves absolute constants, $h^*(\mathcal P)$ and $w^*(\mathcal P)$?
\end{question}
Finally, we have shown that $w^*(\mathcal P)\leq |\mathcal P|$ for any finite poset $\mathcal P$, and equality is achieved. The simplest example when equality is achieved is when $\mathcal P$ is an antichain. It is not hard to construct other examples, but what makes a poset have this property? We therefore ask the following.
\begin{question}
For which finite posets $\mathcal P$ do we have that $w^*(\mathcal P)=|\mathcal P|$?
\end{question}
\vspace{1.5em}
\noindent{\textbf{Acknowledgment.} The first author would like to thank G-Research for generously funding his stay in Cambridge while undertaking this project.}
\bibliographystyle{amsplain}
\bibliography{references}
\Addresses
\end{document}